\documentclass[a4paper,11pt]{article}

\usepackage{a4wide}
\usepackage{amsmath,amssymb,amsthm}
\usepackage{graphicx}
\usepackage{hyperref}
\usepackage{comment}
\usepackage[T1]{fontenc}
\usepackage{listings}
\usepackage{enumitem} 
\usepackage{microtype}
\usepackage{color}
\DeclareMathOperator{\nds}{NDS}

\lstdefinelanguage{Lean}{
  morekeywords={structure, def, theorem, lemma, Prop, Type, Int, Nat, noncomputable, forall, fun, let, by, apply, intro, exact, dsimp, simp, have, match, with, if, then, else, inductive, instance, class, where},
  sensitive=true,
  morecomment=[l]--,
  morestring=[b]",
}
\theoremstyle{definition}
\newtheorem{definition}{Definition}[section]
\newtheorem{example}[definition]{Example}

\theoremstyle{plain}
\newtheorem{theorem}[definition]{Theorem}
\newtheorem{lemma}[definition]{Lemma}

\newtheorem{conjecture}[definition]{Conjecture}

\newcommand{\Ical}{\mathcal{I}}

\title{Average-Rare Order Ideals in Functional Preorders}
\author{Masahiro Hachimori \footnote{Institute of Systems and Information Engineering, University of Tsukuba, Email:hachi@sk.tsukuba.ac.jp} \\
Kenji Kashiwabara\footnote{Corresponding author, Graduate School of Arts and Sciences, The University of Tokyo, Email: cashiwa@g.ecc.u-tokyo.ac.jp}}
\date{}

\begin{document}
\maketitle

\begin{abstract}
We prove that for the preorder induced by a function $f\colon V\to V$, the family of all order ideals is \emph{average-rare}, that is, its normalized degree sum $(\nds)$ is nonpositive.  
As a base case in our reduction, we establish the same result for functional partial orders (or rooted forests).  
%We also propose a conjecture that links this phenomenon to Frankl’s Conjecture. 
We also propose a conjecture related to Frankl’s Conjecture. 
All proofs have been formally verified in the proof assistant \textsc{Lean 4}.
\end{abstract}

\section{Introduction}
In this paper, all the sets, set families, partially ordered sets, and preordered sets are finite.

Frankl’s Conjecture~\cite{Frankl} is a longstanding open problem in extremal set theory.  
The conjecture states that a union-closed family that contains the ground set \(V\) and the empty set must contain an \emph{abundant} element, that is, an element that appears in at least half of the sets of the family.

For a set family \(\mathcal F\subseteq2^{V}\),
we define the \emph{degree} $\deg(u)$ of an element $u\in V$ to be the number of sets containing $u$.
Equivalently, the conjecture states that a union-closed family that contains the ground set \(V\) and the empty set must contain an element $u$ with 
$$ \frac{\deg(u)}{|\mathcal F|} \ge 1/2. $$
%
%
%\paragraph{Previous work.}
Progress on the conjecture has been steady but partial.  
Reimer~\cite{Reimer2003} obtained a general lower bound via combinatorial averaging.  
%Falgas-Ravry~\cite{FalgasRavry2012} settled the conjecture for very small families.  
The conjecture is settled for small families by 
Bo\v{s}njak and P. Markovi\'{c}~\cite{BM},  Vu\v{c}kovi\'{c} and \v{Z}ivkovi\'{c}~\cite{Bojan}, and Roberts and Simpson~\cite{Roberts2010ANO}.

Bruhn and Schaudt~\cite{BruhnSchaudt2015} provide an excellent survey.  
%More recently, Gilmer~\cite{Gilmer2022} introduced an entropy-based approach, which Sawin sharpened to the bound \(\tfrac{3-\sqrt5}{2}\)~\cite{sawin2022improved}.  
%Chase and Lovett obtained optimal constants for approximately union-closed systems~\cite{chase2022approximate}.
More recently, Gilmer~\cite{Gilmer2022} introduced an entropy-based approach and showed a positive constant lower bound, and this was sharpened to
\(\tfrac{3-\sqrt5}{2}\) by several authors~\cite{add1,add2,chase2022approximate,add4,add5,sawin2022improved,add6} as conjectured in \cite{Gilmer2022}.
The authors~\cite{HachimoriKashiwabara2024} gave a systematic study of minimality concepts related to the conjecture.

\bigskip

The conjecture is usually phrased for union-closed families as above, but we work with the dual, intersection-closed form.
%to highlight its link to closure systems.
In the intersection-closed form,
the conjecture states that any intersection-closed family of sets that
contains the ground set \(V\) and the empty set must contain a \emph{rare} element, that is, an element that appears in at most half of the sets of the family.  
Equivalently, the conjecture states that any intersection-closed finite family of sets that contains the ground set $V$ and the empty set contains an element $u$ such that
$$ \frac{\deg(u)}{|\mathcal F|} \le 1/2. $$
The equivalence of the conjecture is straightforward by considering the family of complements
$ \{ V-F \,|\, F\in \mathcal{F} \} $.

Let us consider the average value of $\deg(u)/|\mathcal F|$ over all $u\in V$. 
If this average value is at most $1/2$, we say $\mathcal F$ is \emph{average-rare}. More precisely,
$\mathcal F$ is average-rare if
$$ \frac{\sum_{u\in V}\deg(u)/|\mathcal F|}{|V|} \le \frac{1}{2} .$$
If $\mathcal F$ is average-rare, then the existence of a rare element is deduced.
Hence, the average-rarity is a stronger property than the existence of a rare element.

Since we have 
$$ \sum_{u\in V} \deg(u) = \sum_{F\in \mathcal F}|F| $$
by a double-counting argument, to be average-rare is equivalent to
$$ 2\sum_{F\in \mathcal F}|F| \le |\mathcal F| |V|. $$
We define the
\emph{normalized degree sum} by
\[
  \nds(\mathcal F)
    := 2\sum_{F\in\mathcal F}\!|F|
       \;-\;
       |\mathcal F|\,|V|.
\]
%If \(\nds(\mathcal F)\le0\), we say \(\mathcal F\) is \emph{average-rare}.
By this, $\nds(\mathcal F)\le 0$ if and only if $\mathcal F$ is average-rare.
This NDS is a useful tool in studying average-rarity. Previously, the authors used NDS
in showing the average-rarity of 
``ideal families'' in \cite{hachimori2025averagingproblemidealfamilies}.
(The ``ideal families'' in \cite{hachimori2025averagingproblemidealfamilies} are not families of order ideals, unlike in this paper.)

% Our own contributions include a systematic study of minimality concepts~\cite{HachimoriKashiwabara2024} and an NDS-based analysis of ideal families~\cite{hachimori2025averagingproblemidealfamilies}.

\bigskip

Frankl's conjecture is usually phrased for union-closed families, but in this paper we work with the dual and state it for intersection-closed families.
This is because we intend to work within the framework of the closure systems.

\begin{definition}
A \emph{closure system} is a family of sets that is closed under intersection and contains the ground set.
\end{definition}

Frankl’s conjecture can be equivalently stated as follows: every closure system on a finite set that contains the empty set as a member has a \emph{rare} element.

%We introduce rooted sets because they provide a convenient combinatorial encoding of closure systems.
Closure systems can be represented using rooted sets. Rooted sets provide a convenient combinatorial encoding of closure systems.

\begin{definition}\label{def:rootedfamily}
A \emph{rooted set} is a pair $(A,r)$ consisting of a set $A\subseteq V$ and a
root $r\notin A$. We call $A$ the \emph{stem}. A collection of rooted sets is called a
\emph{family of rooted sets}.
\end{definition}

%In the usual definition of a rooted set one often includes $r$ in $A$, but in this paper we exclude $r$ and call $A$ the stem.
It is better to note that in many contexts a rooted set is defined in such a way that $r$ is included in $A$
(that is, $(A\cup\{r\}, r)$ is called a rooted set in such contexts), but in this paper we exclude the root $r$ from the stem $A$.

%The following lemma is well known and has an easy proof. 
%It is also formally proved by our Lean~4 code in our GitHub repository \cite{KashiwabaraRepo2025}.
The following lemma is well-known.
The proof is straightforward.
\begin{lemma} \label{lemma:generating_rooted_sets}
For a family of rooted sets $\{(A_i,r_i)\}$ on a finite set $V$,
\[
  \bigl\{\,F\subseteq V \;\bigm|\;
          A_i\subseteq F \ \Rightarrow\ r_i \in F \text{ for all } i\in I \bigr\}
\]
is a closure system. Conversely, every closure system admits a representation by a family of rooted sets.
\end{lemma}
%The proof is straightforward.
%(A formal proof by Lean~4 code is provided in our GitHub repository \cite{KashiwabaraRepo2025}.)

The rooted set $(A_i, r_i)$ here means that a set $F$ of the closure system containing $A_i$ must also contain $r_i$. The lemma above states that a closure system is defined by constraints of this type.

Note that a closure system can be represented by different families of rooted sets. When a family $\mathcal A$ of rooted sets represents a closure system $\mathcal F$, we say $\mathcal A$ generates $\mathcal F$.

%In this paper we regard, for a finite ground set $V$, the whole family of ideals of a preordered set as the intersection–closed structure on families of subsets of $V$.

%In particular, if in a generating family of rooted sets each vertex has a unique root and every stem has size~$1$, then a preorder is induced on the vertices: define a binary relation by putting the root on the “smaller” side and the stem on the “larger” side, and take the reflexive–transitive closure. Under this correspondence, the members (hyperedges) of the closure system are exactly the ideals of the resulting preorder.

If in a generating family of rooted sets every stem has size $1$, such as $(\{w\},v)$, then the constraints for the closure system are of the type that a set $F$ containing an element $w$ must contain $v$.
This equivalently means that such a generating family of rooted sets defines  a preorder on $V$:
define a binary relation by $v\lessdot w$ if and only if $(\{w\},v)$ is in the generating family of rooted sets, and take the reflexive-transitive closure of this relation defines a preorder on $V$.
(A \emph{preorder} is a partial order without requiring antisymmetry (i.e., $x\le y$ and $y\le x$ do not imply $x=y$).)
Further, the closure system defined by such
a family of rooted sets corresponds to the
family of order ideals of the preordered set
as follows.

\begin{definition}%[Family of order ideals]
For a preordered set $(V,\le)$, 
$I\subseteq V$ is an \emph{order ideal} if 
$x\in I$ and $y\le x$ implies $y\in I$.
The \emph{family of order ideals} $\Ical(V,\le)$ is
\[
  \Ical(V,\le) := \{\, I \subseteq V \mid x \in I,\ y \le x \Rightarrow y \in I \,\}.
\]
\end{definition}

Note that both the empty set $\varnothing$ and the whole set $V$ always belong to $\Ical(V,\le)$.

\begin{lemma} \label{lem:preorder-ideal}
For a preordered set, the family of order ideals is a closure system. Moreover, the following are equivalent: being representable as a family of order ideals, and being generated by a family of rooted sets all of whose stems are singletons.
\end{lemma}

\begin{proof}
That it is a closure system follows from the fact that the family of all order ideals is closed under intersection and contains the whole set.

If a closure system is represented as the family of order ideals of a preorder, then using the cover relation $x \lessdot y$ one obtains a generating family of rooted sets by taking $\{(\{y\},x)\}$. Conversely, if a closure system is represented by a family of rooted sets with singleton stems, one forms a binary relation by directing each rooted pair $(\{y\},x)$ as $x \le y$; its reflexive–transitive closure is a preorder whose order ideals are exactly the given closure system.
\end{proof}

% This family of order ideals is the class of set families we shall study, with a few additional conditions introduced below.

%\paragraph{Contribution of this paper.}
In this paper, we focus on \emph{functional preorders}, preorders arising from a function \(f\colon V\to V\).  
Given a function $f$, we define a preorder as the reflexive-transitive closure of the covering relations $v\lessdot f(v)$ for each $v\in V$ with $f(v)\neq v$.
Equivalently, a preorder is defined
such that the family of order ideals is generated by the family $\{\, (\{f(v)\}, v) : v\in V, f(v)\neq v\,\}$.
Our Main Theorem (Theorem~\ref{thm:main}) shows that the order-ideal family \(\Ical(V,\le)\) of such a preorder is always average-rare.  
The proof reduces to the Secondary Main Theorem (Theorem~\ref{thm:sub}), which handles the special case where the preorder is a rooted forest (a partial order where each element has at most one cover).  
All arguments have been formally verified in \textsc{Lean 4}, which can be found in our repository~\cite{KashiwabaraRepo2025}.

\bigskip
%\paragraph{Organisation.}
The rest of this paper is structured as follows.
Section~\ref{sec:preorder} introduces functional preorders and states the Main Theorem.  
Section~\ref{sec:secondary} outlines the reduction to rooted forests, and Section~\ref{sec:sub} provides the inductive proof of the Secondary Main Theorem.  
Section~\ref{sec:lean} summarizes the formal verification, and Section~\ref{sec:conclusion} discusses a conjectural extension 
% that would imply 
related to Frankl’s Conjecture.

% ---------------------------------------------------------------------------
% --- Functional Preorders and the Main Theorem -----------------------------
\section{Functional Preorders and the Main Theorem}\label{sec:preorder}

Let \(V\) be a finite set and let \(f\colon V\to V\) be a function on $V$.

\begin{definition}%[Preorder induced by \texorpdfstring{$f$}{f}]
Set \(v\mathrel{\lessdot}w \;\Leftrightarrow\; f(v)=w\) for $f(v)\neq v$. 
The reflexive–transitive closure defines a preorder~\(\le\).
%, where $v\mathrel{\lessdot}w$ is a cover relation.  
We call \((V,\le)\) the \emph{functional preorder} induced by \(f\).
\end{definition}

\begin{lemma}\label{lem:iterates}
For any \(v,w\in V\),
\[
  v\le w \;\Leftrightarrow\; \exists k\ge0 : f^{k}(v)=w,
\]
where \(f^{k}\) denotes the \(k\)-fold composition of \(f\) (with \(f^{0}\) the identity).
\end{lemma}

%\paragraph{Equivalence classes and maximal elements.}
Write \(v\sim w\) when \(v\le w\) and \(w\le v\); 
this defines an equivalence relation on $V$.
If we define a digraph, the \emph{functional graph} of $f$, on $V$ such that an arc is defined from $v$ to $w$ if $f(v)=w$ ($v\neq w$), 
then the equivalence classes are the strongly connected components of the functional graph of~\(f\).

\begin{definition}
In a preorder, an element \(u\in V\) is \emph{maximal} if \(u\le v\) implies \(v\le u\).
\end{definition}

\begin{lemma}\label{lem:max}
In a functional preorder, the elements of any equivalence class of size \(\ge2\) are maximal elements.
\end{lemma}

\begin{proof}
Let \(u\) be an element of an equivalence class of size $\ge 2$ and not maximal.
Since \(u\) is non-maximal, there exist \(v\) and \(u'\) with 
%\(u'\mathrel{\lessdot} v\) 
\(f(u')=v\)
and $u'\sim u$ but \(v\nleq u'\). 
% That is, $f(u) = v$.
On the other hand, there exists $w\in V$ such that $w \neq u'$ and $w\sim u'$, i.e.,  $w \le u'$ and $u' \le w$. That means $f^k(u') = w$ for some $k>1$ by Lemma \ref{lem:iterates}. Then, $f^{k-1}(v) = w$ since $f^k(u')=(f^{k-1}\circ f)(u')=f^{k-1}(v)$. This means $v \le w$, and therefore $v \le u'$, which contradicts $v\nleq u'$.
%Because \(f\) has out-degree 1, this would force two distinct parents for some vertex, contradicting functionality.
\end{proof}

Note that the equivalence class of a maximal element $u$ is a singleton in a functional preorder induced by $f$ if and only if $f(u)=u$.

%A preorder is a partial order if there is no pair of distinct elements $u$ and $v$ with $u\sim v$. 
In a preorder, the condition that there is no pair of distinct elements $u$ and $v$ with $u\sim v$
is equivalent to the condition that the preorder satisfies antisymmetry.
In other words, a preorder is a partial order if every equivalence class is a singleton.
A partially ordered set is called a \emph{poset}.
A functional preorder is a \emph{functional partial order} if it is a partial order.
In a functional partial order induced by $f$, the maximal elements are those $u$ with $f(u)=u$.

\bigskip
% \subsection*{Illustrative examples}
The following three examples show preorders (partial orders) that are functional and not functional, respectively.
These are illustrative examples that indicate functionality is needed for the average-rarity
as in our main theorem.

\begin{example}[A two-element chain]
Let \(V=\{a,b\}\) with \(f(a)=b\) and \(f(b)=b\).  
This induces an order relation \(a<b\).  The order ideals of the preorder (in fact, partial order) $\le$ induced by $f$ are
\[
  \Ical(V,\le)=\{\varnothing,\{a\},\{a,b\}\},
\]
and
\(
  \nds(\Ical(V, \le))
  = 2\cdot (0+1+2)-3\cdot 2
  = 0,
\)
so the family is average-rare.
\end{example}

\begin{example}[The case of a preorder]
Let \(V=\{a,b,c\}\) with \(f(a)=b\), \(f(b)=c\), and \(f(c)=b\).  
Then a preorder is induced by $f$ such that \(a<b\), \(a<c\), and \(b\sim c\).  The order ideals of the preorder are
\[
  \Ical(V,\le)=\{\varnothing,\{a\},\{a,b,c\}\},
\]
and
\(
  \nds(\Ical(V, \le))
  = 2\cdot (0+1+3)-3\cdot 3
  = -1 < 0,
\)
so the family is average-rare.
\end{example}

\begin{example}[A non-functional rooted-set family]
\label{example:not_average_rare}
Take \(V=\{a,b,c\}\) and rooted sets \((\{b\},a)\), \((\{c\},a)\).  
% The resulting ideal family has five hyperedges of total size 8, 
This determines preorder (in fact, partial order) $\le$ induced by the two covering relations $a\lessdot b$ and $a\lessdot c$.
The resulting ideal family is
\[ \Ical(V,\le) =\{\varnothing, \{a\}, \{a,b\}, \{a,c\}, \{a,b,c\} \}\]
giving
\(
%  \nds = 2\cdot8-5\cdot3 = 1>0,
  \nds(\Ical(V,\le)) = 2\cdot (0+1+2+2+3)-5\cdot 3 = 1>0,
\)
hence it is \emph{not} average-rare. 
This is not functional because the element $a$ has two distinct covers, $b$ and $c$.
\end{example}

%\subsection*{Statement of the main result}
%The following is the statement of the main result.
The following is our main theorem.

\begin{theorem}[Main Theorem]\label{thm:main}
For any finite set \(V\) and function \(f\colon V\to V\),  
the order-ideal family \(\Ical(V,\le)\) of the induced functional preorder satisfies
\[
  \nds\bigl(\Ical(V,\le)\bigr)\le0.
\]
\end{theorem}

% The proof proceeds by reducing to rooted forests (the Secondary Main Theorem (Theorem~\ref{thm:sub}) and Lemma \ref{lem:forest}) in Section \ref{sec:secondary}, then applying an induction on \(|V|\) (Section~\ref{sec:sub}).  In rooted forests, maximal vertices are provably rare, and tracing out non-trivial equivalence classes never increases NDS—mechanisms made precise in Lemmas \ref{lem:trace}–\ref{lem:tracends}.

\bigskip
We prove the Main Theorem (Theorem~\ref{thm:main}) by reducing to the following Secondary Main Theorem (Theorem~\ref{thm:sub}).

\begin{theorem}[Secondary Main Theorem]\label{thm:sub}
Let $(V,\le)$ be a functional \emph{partial order}.
%—equivalently, a rooted forest.  
Then
\[
  \nds\bigl(\Ical(V,\le)\bigr)\le0.
\]
\end{theorem}

The reduction of the Main Theorem to the Secondary Main Theorem is given in Section~\ref{sec:secondary}, and 
the Secondary Main Theorem
will be proved in Section~\ref{sec:sub}.

We say a poset is a \emph{rooted forest} if its Hasse diagram is acyclic and each connected component has a unique maximal element,
where the unique maximal element is the \emph{root} of the component.
Here, we remark that a poset is functional
if and only if it is a rooted forest,
as shown in the following lemma.
Hence, the Secondary Main Theorem can be equivalently stated that the family of order ideals of a rooted forest is average-rare.

\begin{lemma}\label{lem:forest}
A partial order on a finite set is functional if and only if it is a rooted forest, i.e.,  its Hasse diagram is acyclic (as an undirected graph) and each connected component has a unique maximal element.
\end{lemma}

\begin{proof}
Consider the Hasse diagram of the poset as a directed graph by orienting each edge of the Hasse diagram from the lower vertex to the upper vertex.
By this, the Hasse diagram of a poset will be an acyclic digraph (i.e., a digraph with no directed cycles).

Suppose first that $(V,\le)$ is functional. 
% that is, 
Then
every element has 
%at most one immediate successor (cover) in its Hasse diagram.  
at most one outgoing edge in its Hasse diagram.
If the Hasse diagram contains a (undirected) cycle, then, 
% since $V$ is finite, 
each vertex on the cycle has exactly one outgoing edge to another vertex on the cycle, forming a strongly connected component of size at least two. 
This contradicts the antisymmetry of the partial order.  
Hence, the Hasse diagram is acyclic as an undirected graph.

Moreover, in a finite acyclic digraph,
%where every vertex has out-degree at most one
each connected component must contain a sink, a vertex of out-degree zero.  
Such a vertex is a maximal element of that component.  
%If there were two distinct maximal elements in the same component, they could not be comparable, and there would be no path between them in the Hasse diagram, contradicting connectedness of the component.  
If there are two distinct maximal elements in the same component, there is an undirected path connecting the two vertices on the Hasse diagram. Then, a minimal element among the vertices of the path is found on an internal vertex of the path, % since otherwise it contradicts the fact that the two end vertices are maximal elements.
since the end vertices are maximal elements and cannot be minimal.
This minimal element must have 
% two different successors, 
two outgoing edges,
contradicting the functionality.
Thus, each connected component has a unique maximal element.

Conversely, assume that the Hasse diagram is acyclic as an undirected graph and that each connected component has a unique maximal element $r$.  
Acyclicity implies that each component is a finite tree.
%oriented toward its maximal element.  
In each tree, from every vertex there exists a path to $r$ going upward, since $r$ is the unique maximal element in the component. Since there exists only one path between two vertices in a tree, this shows that each vertex has at most one 
% immediate successor 
outgoing edge
in the Hasse diagram.  
Therefore, the partial order is functional.
This proves the equivalence.
\end{proof}

If the Hasse diagram of a poset $(V,\le)$ has two or more connected components, then the poset can be expressed as a disjoint union $(C_1, \le_1) + (C_2,\le_2)$ with $V=C_1\sqcup C_2$, where $\le_1$ and $\le_2$ are the restrictions of the order relation $\le$ to $C_1$ and $C_2$, respectively.
(The disjoint union $(C_1, \le_1) + (C_2,\le_2)$ is the poset on $V=C_1\sqcup C_2$ with the order relation $\le$ such that $x\le y$ if $x,y\in C_1$ and $x\le_1  y$, or if $x,y\in C_2$ and $x\le_2 y$.)

% ---------------------------------------------------------------------------
% --- Secondary Main Theorem and Reduction ----------------------------------
\section{Reduction to Secondary Main Theorem}\label{sec:secondary}

In this section, we reduce the proof of the Main Theorem (Theorem~\ref{thm:main}) to that of the Secondary Main Theorem (Theorem~\ref{thm:sub}). 
% Since the proof of the Secondary Main Theorem is the base case of the Main Theorem, we show the inductive case in this section.

% -- Rare element for maximal vertices --------------------------------------
\begin{lemma}\label{lem:rare}
For any maximal element $u$ of a preordered set $(V, \le)$, $u$ is a rare element in
$\Ical(V,\le)$.
\end{lemma}

\begin{proof}
Let $U$ be the equivalence class containing $u$.  Define a map
\[
  \Phi : \{\,I\in\Ical(V,\le)\mid u\in I\,\} \longrightarrow
         \{\,J\in\Ical(V,\le)\mid u\notin J\,\}%,
%  \qquad
%  \Phi(I) := I \setminus U.
\]
by 
\[ \Phi(I) := I \setminus U. \]
Here, $\Phi(I)$ is an order ideal not containing $u$ by maximality of $u$,
and if
$\Phi(I_1)=\Phi(I_2)$ then $I_1=I_2$,
since every element of $U$ appears 
% in exactly the same ideals.  
in the same ideals simultaneously.
This implies that $\Phi$ is injective, and hence,
the number of ideals containing $u$ is at most the number not containing $u$.
Therefore, $\deg_{\Ical}(u) \le \tfrac12 |\Ical(V,\le)|$, i.e., $u$ is rare.
\end{proof}
% ---------------------------------------------------------------------------

%We now explain how the Main Theorem is reduced to a rooted-forest (or partial order) statement, the
%\emph{Secondary Main Theorem}.  Three ingredients are involved:

%\begin{enumerate}
%  \item[\emph{(A)}] \emph{Parallel‐element tracing.}  
%        Collapsing an equivalence class of size $\ge2$ by a trace operation
%        never increases the normalised degree sum (NDS).
%  \item[\emph{(B)}] \emph{Forest structure.}  
%        After repeatedly tracing all non-trivial classes, the preorder becomes
%        a rooted forest: a partial order whose Hasse diagram is acyclic and has
%        a unique maximal element in each connected component.
%  \item[\emph{(C)}] \emph{Rooted forests are average-rare.}  
%        Section~\ref{sec:sub} proves that every rooted-forest poset satisfies
%        $\nds(\Ical)\le0$.
%\end{enumerate}

%Combining (A)–(C) yields the Main Theorem.

%\subsection*{Parallel elements and the trace operation}

\begin{definition}%[Parallel elements]
Given a set family $\mathcal F$ on $V$, two elements $u,v\in V$ are
\emph{parallel in $\mathcal F$} if
\[
  \{F\in\mathcal F\mid u\in F\}
  =
  \{F\in\mathcal F\mid v\in F\}.
\]
We say $v$, different from $u$, is a \emph{parallel partner} of $u$ if $u$ and $v$ are parallel.
\end{definition}

The following lemma is straightforward.
\begin{lemma}\label{lem:parallel-equivalence}
For the family of order ideals $\Ical(V,\le)$ for any preorder,  
$u\sim v$ (i.e., both $u\le v$ and $u\ge v$) if and only if $u$ and $v$ are parallel in $\Ical(V,\le)$.
\end{lemma}

\bigskip
In the reduction, we use the following operator to the family of order ideals of the functional preorder.
\begin{definition}%[Trace]
\label{def:trace}
For $x\in V$ and a family $\mathcal F\subseteq2^{V}$, 
we define the \emph{traced} family at $x$ by
\[
  \operatorname{trace}_x(\mathcal F)
    := \{\,F\setminus\{x\}\mid F\in\mathcal F\,\}.
\]
\end{definition}

The following is the key lemma in the reduction.
\begin{lemma}%[Injectivity of trace]
\label{lem:trace}
If $u$ has a parallel partner in $\mathcal F$,  
the \emph{trace map} 
\[ \Phi_u:{\mathcal F} \rightarrow \operatorname{trace}_u(\mathcal F) \]
defined by 
$F\mapsto F\setminus\{u\}$ is injective 
% on the hyperedges of $\mathcal F$.
\end{lemma}

\begin{proof}
Assume $I_1\setminus\{u\}=I_2\setminus\{u\}$.  If $I_1\ne I_2$ they differ only
by $u$.  But then their memberships of the parallel partner $v$ would also
differ, contradicting parallelism.  Hence $I_1=I_2$.
\end{proof}

This implies that, if $u$ has a parallel partner,
then tracing preserves the number of sets and all degrees except that
of $u$.  
% Consequently:
This is the key property to show the following lemma.

\begin{lemma}\label{lem:tracends}
Let $(V,\le)$ be a functional preorder and let $u$ lie in an equivalence class of size at least $2$.  
Write $V' := V \setminus \{u\}$.  
Then:
\begin{enumerate}
  \item[(i)] 
  % the traced family $\Ical(V',\le)$ arises from a functional preorder, and
  %\color{red}
  there is a function $g:V'\rightarrow V'$ such that the preorder $\le'$ on $V'$ induced by $g$ has 
  \( \Ical(V', \le') = \operatorname{trace}_u(\Ical(V, \le))\)
  (the preorder $\le'$ is functional induced by $g$),
  and
  %\color{black}
  \item[(ii)] $\displaystyle \nds\bigl(\Ical(V,\le)\bigr) \le \nds\bigl(\Ical(V',\le')\bigr)$.
\end{enumerate}
\end{lemma}

Remark that (i) of the lemma states that $(V',\le')$ is a functional preorder.

\begin{proof}
% \textbf{(1) Preservation of functionality under tracing.}  
\mbox{}
\begin{enumerate}
\item[(i)]
\noindent 
\begin{comment}
Let $(V,\le)$ be a functional preorder, so every element has at most one immediate successor in the Hasse diagram.  
The trace operation removes $u$ from all ideals and restricts the order to the vertex set $V' := V \setminus \{u\}$.  
Formally, the induced order $\le'$ on $V'$ is the restriction of $\le$ to $V' \times V'$.  
If $x,y \in V'$ are such that $x \lessdot' y$ in the Hasse diagram of $(V',\le')$, then $x \lessdot y$ in $(V,\le)$ and both $x$ and $y$ are different from $u$.  
Since $(V,\le)$ is functional, $x$ has at most one such $y$ in $V$, hence at most one in $V'$.  
Thus $(V',\le')$ is also functional.
\end{comment}
%\color{red}
Let the preordered set $(V,\le)$ be induced by $f:V\rightarrow V$, and
let $g: V'\rightarrow V'$ be defined by $g(x)=f(x)$ when $f(x)\neq u$ and $g(x)=f^2(x)$ when $f(x)=u$.
It can be easily verified that the preordered set $(V', \le')$ induced by this $g$ is a restriction of $(V,\le)$ on $V'$. This $g$ is a required function.
Indeed, if $I\in \Ical(V, \le)$, then it is easy to see that $I\setminus \{u\}$ is an order ideal in $\Ical(V', \le')$.
On the other hand, if $I'\in \Ical(V', \le')$, then by letting
\[ I = \begin{cases} I'\cup\{u\} & \text{if $f(u)\in I'$,} \\ I' & \text{otherwise}. \end{cases} \]
$I$ is an order ideal in $\Ical(V, \le)$ and $I\setminus \{u\} = I'$.
Hence, $\Ical(V', \le') = \operatorname{trace}_u(\Ical(V, \le))$.
%\color{black}

%\medskip
%\textbf{(2) Inequality for the normalised degree sum.}  
\item[(ii)]
By Lemma~\ref{lem:trace}, the trace map 
\[
  \Phi_u : \Ical(V,\le) \to \Ical(V',\le'), \quad I \mapsto I \setminus \{u\}
\]
is injective when $u$ has a parallel partner.  
Therefore $|\Ical(V,\le)| = |\Ical(V',\le')|$.

% For the sum of ideal sizes, note that removing $u$ from all ideals decreases the total size sum exactly by $\deg(u)$, the number of ideals containing $u$.  
% Since $|V| = |V'| + 1$, we compute:
Since $\Ical(V',\le')$ is the result of removing $u$ from all ideals in $\Ical(V,\le)$, we have
\[ \sum_{I\in \Ical(V,\le)} |I| = \sum_{I'\in \Ical(V',\le')} |I'|+\deg(u),\]
and we have $|V| = |V'|+1$. Hence,
\begin{align*}
\nds(\Ical(V,\le))
 &= 2\cdot\sum_{I \in \Ical(V,\le)} |I| - |\Ical(V,\le)|\,|V| \\
 &= 2\left( \sum_{I' \in \Ical(V',\le')} |I'| + \deg(u) \right) - |\Ical(V',\le')|\,(|V'| + 1) \\
 &= \left[ 2\cdot\sum_{I' \in \Ical(V',\le')} |I'| - |\Ical(V',\le')|\,|V'| \right] + 2\deg(u) - |\Ical(V,\le)| \\
 &= \nds(\Ical(V',\le')) + 2\deg(u) - |\Ical(V,\le)|.
\end{align*}

By Lemma~\ref{lem:max}, $u$ is a maximal element of $(V,\le)$, and by Lemma~\ref{lem:rare}, $u$ is rare in $\Ical(V,\le)$.  
The rarity of $u$ gives $2\deg(u) \le |\Ical(V,\le)|$, and substituting this into the above expression yields
\[
\nds(\Ical(V,\le)) \le \nds(\Ical(V',\le')),
\]
as desired.
\end{enumerate}
\end{proof}
\begin{comment}
\begin{lemma}\label{lem:tracends}
Let $(V,\le)$ be a functional preorder and let
$u$ lie in an equivalence class of size $\ge2$.  
Write $V'=V\setminus\{u\}$.  Then
\begin{enumerate}
  \item the traced family $\Ical(V',\le)$ arises from a functional preorder, and
  \item $\displaystyle\nds\bigl(\Ical(V,\le)\bigr)\le
         \nds\bigl(\Ical(V',\le)\bigr).$
\end{enumerate}
\end{lemma}

\begin{proof}[Sketch]
Injectivity (Lemma~\ref{lem:trace}) gives
$|\Ical(V,\le)|=|\Ical(V',\le)|$.  The degree of every vertex
other than $u$ is unchanged; the degree of $u$ disappears, lowering the total
ideal–size sum by $\deg(u)$.  A short calculation shows
\[
  \nds(\Ical(V,\le))
    = \nds(\Ical(V',\le))
      + 2\deg(u) - |\Ical(V,\le)|.
\]
Since $u$ is maximal (Lemma~\ref{lem:max}) and rare
(Lemma~\ref{lem:rare}), $2\deg(u)\le|\Ical(V,\le)|$, yielding the inequality.
\end{proof}
\end{comment}

%\subsection*{Proof of Secondary Main Theorem}

% Secondary Main Theorem was here previously

\bigskip
\noindent
\begin{proof} [Proof of the Main Theorem from the Secondary Main Theorem. ]
\mbox{}\\
% The proof is an induction on $|V|$ carried out in Section~\ref{sec:sub}.  
With Lemma~\ref{lem:tracends} in hand, repeatedly
tracing out nontrivial equivalence classes reduces
the size of each equivalence class to one.
This reduces the Main Theorem (Theorem~\ref{thm:main}) 
%(for arbitrary functional preorders) 
to the Secondary Main Theorem (Theorem~\ref{thm:sub})
since a preorder such that each equivalence class is a singleton is a partial order.
%\hfill $\Box$
\end{proof}

What remains is the proof of the Secondary Main Theorem (Theorem~\ref{thm:sub}), which is given in Section~\ref{sec:sub}.

% ---------------------------------------------------------------------------
% --- Proof of the Secondary Main Theorem -----------------------------------
\section{Proof of the Secondary Main Theorem}\label{sec:sub}

We now prove the Secondary Main Theorem (Theorem~\ref{thm:sub}), namely,
$\nds\bigl(\Ical(V,\le)\bigr)\le 0$
for every rooted-forest poset $(V,\le)$.
The argument proceeds by induction on the number of vertices~$n=\lvert V\rvert$,
using four technical lemmas.

%\subsection{Removing the unique root}
For the following lemma, note that a connected rooted-forest poset, i.e., a rooted-forest poset such that the Hasse diagram is connected, has a unique maximal element by Lemma~\ref{lem:forest}.

\begin{lemma}%[Root deletion]
\label{lem:deletion}
Let $(V,\le)$ be a \emph{connected} rooted-forest poset. Let the unique maximal
element be $x$, and denote $V' := V\setminus\{x\}$. 
Then $ (V',\le') $ is a
functional poset,
where $\le'$ is a partial order restricting the order relation $\le$ to $V'$,
and
\[
  \nds\bigl(\Ical(V,\le)\bigr)
  \;=\;
  \nds\bigl(\Ical(V',\le')\bigr)
  \;+\;
  \bigl(\lvert V\rvert - \lvert\Ical(V,\le)\rvert + 1\bigr).
\]
\end{lemma}

\begin{proof}
Since $(V, \le)$ is a connected rooted forest, there is a function $f:V\rightarrow V$ that induces the partial order, with $f(x)=x$ for the unique maximal element $x$.

Let $g:V'\rightarrow V'$ be defined by
\[
  g(v)=\begin{cases} f(v) & \text{if $f(v)\neq x$,} \\
  v & \text{if $f(v)=x$.} \end{cases}
\]
Then it is easily verified that the order relation induced by $g$ is exactly the order relation $\le'$ that restricts the order relation induced by $g$ on $V'=V\setminus\{x\}$.
Hence $(V',\le')$ is a functional poset.
%\color{black}

\medskip
\begin{comment}
\textbf{(B) Classification of ideals with respect to $x$.}
Because $(V,\le)$ is connected with unique maximal element $x$, every
$y\in V$ satisfies $y\le x$. Let $I\in\Ical(V,\le)$.
\begin{itemize}
  \item If $x\in I$, then $I$ is downward closed, hence contains every
        $y\le x$, so $I=V$.
  \item If $x\notin I$, then $I\subseteq V'$ and $I$ is downward closed
        in $(V',\le')$, i.e.\ $I\in\Ical(V',\le')$.
\end{itemize}
Conversely, any $J\in\Ical(V',\le')$ is an ideal of $(V,\le)$ not containing
$x$. Thus inclusion yields a bijection
\[
  \Ical(V',\le') \;\longleftrightarrow\;
  \{\,I\in\Ical(V,\le)\mid x\notin I\,\},
\]
and therefore
\[
  \lvert\Ical(V,\le)\rvert \;=\; \lvert\Ical(V',\le')\rvert + 1,
  \qquad
  \sum_{I\in\Ical(V,\le)} \lvert I\rvert
  \;=\;
  \sum_{I'\in\Ical(V',\le')} \lvert I'\rvert \;+\; \lvert V\rvert,
\]
since the only ideal of $(V,\le)$ containing $x$ is $V$ itself.
\end{comment}

%\color{red}
Since the only difference of $(V, \le)$ and $(V',\le')$ is the element $x$, the number of order ideals not containing $x$ is the same.
%Order ideals that contain $x$ exist in $(V, \le)$, and there is only one such order ideal
%(i.e., the whole $V$) since $x$ is the unique maximal element in $(V,\le)$. 
The only order ideal containing $x$ is the whole $V$ in $\Ical(V,\le)$ since $x$ is the unique maximal element in $(V,\le)$.
Therefore, we have
\[
  \lvert\Ical(V,\le)\rvert \;=\; \lvert\Ical(V',\le')\rvert + 1,
  \qquad
  \sum_{I\in\Ical(V,\le)} \lvert I\rvert
  \;=\;
  \sum_{I'\in\Ical(V',\le')} \lvert I'\rvert \;+\; \lvert V\rvert.
\]
%\color{black}
 
% \medskip
% \textbf{(C) Computing the NDS.}
Let $n:=\lvert V\rvert$, $n':=\lvert V'\rvert = n-1$,
$m':=\lvert\Ical(V',\le')\rvert$, $m:=\lvert\Ical(V,\le)\rvert = m'+1$, and
\[
  S' := \sum_{I'\in\Ical(V',\le')} \lvert I'\rvert,
  \qquad
  S := \sum_{I\in\Ical(V,\le)} \lvert I\rvert \;=\; S' + n .
\]
Then
\begin{align*}
  \nds(\Ical(V,\le)) - \nds(\Ical(V',\le'))
    &= \bigl(2(S'+n) - (m'+1)n\bigr) - \bigl(2S' - m'n'\bigr) \\
    &= 2n - (m'+1)n + m'n' \\
    &= n - m'(n-n')
     \;=\; n - m' \\
    &= \lvert V\rvert - \lvert\Ical(V',\le')\rvert
     \;=\; \lvert V\rvert - \bigl(\lvert\Ical(V,\le)\rvert - 1\bigr) \\
    &= \lvert V\rvert - \lvert\Ical(V,\le)\rvert + 1 .
\end{align*}
Rearranging yields the stated identity.
\end{proof}

\begin{comment}
\begin{proof}
Because \(x\) is above every other vertex, the only ideal containing \(x\) is
\(V\) itself.  Hence
\(
  \lvert\Ical(V,\le)\rvert
  = \lvert\Ical(V',\le)\rvert + 1
\)
and
\(
  \sum_{I\in\Ical(V,\le)}\!\lvert I\rvert
  = \sum_{I\in\Ical(V',\le)}\!\lvert I\rvert + \lvert V\rvert.
\)
Since \(\lvert V\rvert=\lvert V'\rvert+1\), substituting into the definition of
\(\nds\) yields the stated identity.
\end{proof}
\end{comment}

%\subsection{A lower bound on the number of ideals}

\begin{lemma}\label{lem:lower_bound}
For any finite poset \((V,\le)\),
\(
  \lvert\Ical(V,\le)\rvert \ge \lvert V\rvert + 1.
\)
\end{lemma}

\begin{proof}
% The injection \(v\mapsto\downarrow v\) (principal ideal) gives
Let $I_v =\{u\in V: u\le v\}$, i.e., the principal ideal of $v$. The injection \(v\mapsto I_v\) gives
\(\lvert V\rvert\) mutually distinct non-empty ideals by antisymmetry.
Adding the empty ideal completes the bound.
\end{proof}

\begin{lemma}%[NDS of a disjoint union]
\label{lem:unionsum}
For set families $\mathcal F_1,\mathcal F_2$ (on a common finite ground set $U$),
\[
  \sum_{A\in\mathcal F_1}\sum_{B\in\mathcal F_2}\lvert A\cup B\rvert
  \;=\;
    \lvert\mathcal F_2\rvert\sum_{A\in\mathcal F_1}\lvert A\rvert
    \;+\;
    \lvert\mathcal F_1\rvert\sum_{B\in\mathcal F_2}\lvert B\rvert
    \;-\;
    \sum_{A\in\mathcal F_1}\sum_{B\in\mathcal F_2}\lvert A\cap B\rvert.
\]
\end{lemma}

\begin{proof}
% Fix a finite ground set $U$ with $\mathcal F_1,\mathcal F_2\subseteq 2^U$.
For each $A\subseteq U$ and $B\subseteq U$ we have the identity
$\lvert A\cup B\rvert = \lvert A\rvert + \lvert B\rvert - \lvert A\cap B\rvert$.
Summing over all ordered pairs $(A,B)\in \mathcal F_1\times\mathcal F_2$,
% and using linearity of finite sums,
\begin{comment}
\begin{align*}
\sum_{A\in\mathcal F_1}\sum_{B\in\mathcal F_2} \lvert A\cup B\rvert
 &= \sum_{A}\sum_{B}\bigl(\lvert A\rvert + \lvert B\rvert - \lvert A\cap B\rvert\bigr)\\
 &= \underbrace{\sum_{A}\sum_{B}\lvert A\rvert}_{\text{(I)}}
    \;+\;
    \underbrace{\sum_{A}\sum_{B}\lvert B\rvert}_{\text{(II)}}
    \;-\;
    \underbrace{\sum_{A}\sum_{B}\lvert A\cap B\rvert}_{\text{(III)}}.
\end{align*}
We evaluate the three terms separately.

\smallskip
\emph{Term (I).} For fixed $A\in\mathcal F_1$, the inner sum over $B$ simply
repeats $\lvert A\rvert$ exactly $\lvert\mathcal F_2\rvert$ times:
$\sum_{B\in\mathcal F_2}\lvert A\rvert=\lvert\mathcal F_2\rvert\,\lvert A\rvert$.
Therefore
\[
  \sum_{A\in\mathcal F_1}\sum_{B\in\mathcal F_2}\lvert A\rvert
  = \lvert\mathcal F_2\rvert \sum_{A\in\mathcal F_1}\lvert A\rvert.
\]

\smallskip
\emph{Term (II).} Symmetrically,
$\sum_{A\in\mathcal F_1}\sum_{B\in\mathcal F_2}\lvert B\rvert
 = \lvert\mathcal F_1\rvert \sum_{B\in\mathcal F_2}\lvert B\rvert$.

\smallskip
\emph{Term (III).} We keep it as the double sum
$\sum_{A\in\mathcal F_1}\sum_{B\in\mathcal F_2}\lvert A\cap B\rvert$.

\smallskip
Substituting these evaluations into the decomposition above yields exactly the
claimed identity.

\end{comment}

\begin{align*}
\sum_{A\in\mathcal F_1}\sum_{B\in\mathcal F_2} \lvert A\cup B\rvert
 &= \sum_{A\in\mathcal F_1}\sum_{B\in\mathcal F_2}\bigl(\lvert A\rvert + \lvert B\rvert - \lvert A\cap B\rvert\bigr)\\
 &= \sum_{A\in\mathcal F_1}\sum_{B\in\mathcal F_2}\lvert A\rvert
    \;+\;
    \sum_{A\in\mathcal F_1}\sum_{B\in\mathcal F_2}\lvert B\rvert
    \;-\;
    \sum_{A\in\mathcal F_1}\sum_{B\in\mathcal F_2}\lvert A\cap B\rvert 
    \\
&= 
    \lvert\mathcal F_2\rvert\sum_{A\in\mathcal F_1}\lvert A\rvert
    \;+\;
    \lvert\mathcal F_1\rvert\sum_{B\in\mathcal F_2}\lvert B\rvert
    \;-\;
    \sum_{A\in\mathcal F_1}\sum_{B\in\mathcal F_2}\lvert A\cap B\rvert.
\end{align*}

\end{proof}

\begin{lemma}%[Product formula]
\label{lem:product}
% If a rooted-forest poset decomposes as a disjoint union of two non-empty
%components \(C_1\sqcup C_2\), 
If a poset $(V,\le)$ decomposes as a disjoint union of two non-empty posets, 
$(V,\le) = (C_1,\le_1) + (C_2,\le_2)$, 
then
\[
  \nds\bigl(\Ical(V,\le)\bigr)
  = \lvert\Ical(C_2,\le_2)\rvert\,\nds\!\bigl(\Ical(C_1,\le_1)\bigr)
    + \lvert\Ical(C_1,\le_1)\rvert\,\nds\!\bigl(\Ical(C_2,\le_2)\bigr).
\]
\end{lemma}

\begin{proof}
\begin{comment}
The bijection \(\Ical(V,\le)\!\to\!\Ical(C_1,\le)\times\Ical(C_2,\le)\) given by
\(I\mapsto(I\cap C_1,I\cap C_2)\) preserves counts; the double sum of
ideal-sizes separates via Lemma~\ref{lem:unionsum}.
\end{comment}
Since $(C_1,\le_1)$ and $C_2,\le_2)$ are disjoint,
an order ideal $I$ of $(V,\le)$ is decomposed as $I=I_1\sqcup I_2$ with $I_1\in \Ical(C_1,\le_1)$ and $I_2\in\Ical(C_2,\le_2)$.
Hence,
$$ |\Ical(V,\le)| = |\Ical(C_1,\le_1)|\cdot |\Ical(C_2,\le_2)|. $$
The statement follows from Lemma~\ref{lem:unionsum}.
\end{proof}

%\subsection{Induction on \texorpdfstring{$|V|$}{|V|}}

\bigskip
% \begin{proof}[Proof of Theorem~\ref{thm:sub}]
\begin{proof}[Proof of the Secondary Main Theorem (Theorem~\ref{thm:sub})]
\mbox{}\\
The proof is by induction on \(n=\lvert V\rvert\).

% \emph{Base case \(n=1\).}  Then
When $n=1$, 
\(\Ical=\{\varnothing,V\}\) and
\(\nds(\Ical)=0\).

% \emph{Inductive step.}
When $n\ge 2$,
assume the theorem is true for all rooted forests with fewer than \(n\) vertices.

\begin{enumerate}
%   \item[\textbf{(a)}] \emph{Disconnected case.}  
%        Write \(V=C_1\sqcup C_2\) with both \(C_i\neq\varnothing\).  
%        By induction each component has non-positive NDS; multiplying by the
%        positive factors in Lemma~\ref{lem:product} preserves the sign.
\item[-]
If $(V,\le)$ is disconnected, then it can be expressed by a disjoint union as
$(V,\le) = (C_1,\le_1) \linebreak[0] + (C_2,\le_2)$ with $V=C_1\sqcup C_2$, and each $(C_i,\le_i)$ has non-positive NDS by the induction hypothesis.
Hence $\nds(\Ical(V,\le))$ is non-positive by Lemma~\ref{lem:product} and the non-positivity of $\nds(\Ical(C_1,\le_1))$ and $\nds(\Ical(C_2,\le_2))$.

%  \item[\textbf{(b)}] \emph{Connected case.}  
%        Let \(x\) be the unique maximal vertex by Lemma \ref{lem:forest}.  
%        Delete \(x\) to obtain \(V'\).  
%        The induction hypothesis gives
%        \(\nds(\Ical(V',\le))\le0\).  
%        Lemmas \ref{lem:deletion} and \ref{lem:lower_bound} imply
%        \(\nds(\Ical(V,\le))\le\nds(\Ical(V',\le))\le0\).
\item[-]
If $(V,\le)$ is connected, let $x$ be the unique maximal element.
Let $V'=V\setminus\{x\}$ and $\le'$ be the restriction of $\le$ to $V'$.
By Lemma~\ref{lem:deletion}, Lemma~\ref{lem:lower_bound}, and the induction hypothesis, we have
\begin{align*}
\nds(\Ical(V,\le))
&=\nds(\Ical(V',\le'))+(|V|-|\Ical(V,\le)|+1) \\
&\le \nds(\Ical(V',\le'))\le0 .
\end{align*}
\end{enumerate}
%Thus \(\nds(\Ical(V,\le))\le0\) for all \(n\).
\end{proof}

% ---------------------------------------------------------------------------
% --- Formal Verification in Lean 4 ----------------------------------------
%\section{Formal Verification in Lean 4}\label{sec:lean}

\section{Lean 4 Formalization}
\label{sec:lean}
The validity of all theorems and lemmas in this paper is machine-checked in
the proof assistant \textsc{Lean 4}.
%\subsection{Why formalise?}
Using proof assistants offers several advantages.
First, they ensure the rigor of the proofs, as every inference is validated by Lean’s kernel, eliminating overlooked edge cases and computational errors.
They also ensure reproducibility. The proof scripts are plain text and can be replayed by anyone to obtain a formal verification.
Moreover, the codes used for the proofs can be reused by others, as the relevant definitions and lemmas are incorporated into the community library \texttt{mathlib4}, thereby accelerating future work on Frankl-type problems.

 We briefly describe the formalized code used in our proofs. % and the AI-assisted workflow that made itpractical.
The Lean code snippets below are simplified for readability and exposition. For the complete and precise formalization, please refer to the source code in the repository 
\verb+https://github.com/kashiwabarakenji/avg-rare+.

% --- SetFamily ---
\subsection*{SetFamily}
\begin{lstlisting}
structure SetFamily  where
  ground : Finset $\alpha$
  sets : Finset $\alpha$ $\rightarrow$ Prop
\end{lstlisting}

Structure \texttt{SetFamily} provides a family of subsets of a finite ground set.
\texttt{SetFamily.sets} is the predicate that determines whether a given subset is in the family.

% --- NDS ---
\subsection*{Normalized Degree Sum (NDS)}
\begin{lstlisting}
def NDS (F : SetFamily $\alpha$) : Int :=
  2 * F.totalHyperedgeSize - F.numHyperedges * F.ground.card
\end{lstlisting}

The above code is the definition of the normalized degree sum (NDS):
\[
\mathrm{NDS}(\mathcal{F}) := 2 \sum_{F \in \mathcal{F}} |F| - |\mathcal{F}| \cdot |V|.
\]
NDS measures the average-rarity of elements in a set family. When NDS $\le 0$, the family is average-rare.
\texttt{F.ground.card} means the cardinality of the ground set. \texttt{Int} represents the set of all integers.
In the above, NDS returns an integer value.

% --- FuncSetup ---
\subsection*{FuncSetup}
\begin{lstlisting}
structure FuncSetup ($\alpha$ : Type) where
  ground : Finset $\alpha$
  f : ground $\rightarrow$ ground
\end{lstlisting}

Structure \texttt{FuncSetup} provides the assumption for our problem, including the ground set and the function $f:V \rightarrow V$ on the ground set.
It induces the preorder \texttt{FuncSetup.le} on the ground set.

% --- Order Ideal ---
\subsection*{Order Ideal}
\begin{lstlisting}
def isOrderIdealOn (S: FuncSetup $\alpha$) (le : $\alpha$ $\rightarrow$ $\alpha$ $\rightarrow$ Prop) (I : Finset $\alpha$) : Prop :=
  I $\subseteq$ S.ground $\land$
  $\forall x, x \in I \rightarrow \forall y,  y\in$ S.ground $\rightarrow$ S.le y x $\rightarrow$ y $\in$ I
\end{lstlisting}
This is the condition of the order ideal
\[
I \subseteq V \quad \wedge \quad (\forall x \in I, \forall y \in V,\ y \leq x \Rightarrow y \in I).
\]
\texttt{FuncSetup.idealFamily} is the set of all order ideals defined using \texttt{isOrderIdealOn}.

% --- Main Theorem ---
\subsection*{Main Theorem}
\begin{lstlisting}
theorem main_nds_nonpos {$\alpha$ : Type} 
  (S : FuncSetup $\alpha$) :
  (S.idealFamily).NDS $\leq$ 0 := by
  apply Reduction.main_nds_nonpos_of_secondary
  intro T hT
  have hT' : isPoset T := by
    dsimp [isPoset]
    dsimp [has_le_antisymm]
    exact T.antisymm_of_isPoset hT
  exact secondary_main_theorem T hT'
\end{lstlisting}
This is the Main Theorem of our paper: for any function \(f:V\rightarrow V\), the induced order ideal family $\mathcal{I}$ is always average-rare:
\[
{\mathrm{NDS}(\mathcal{I}(V, \le)) \le 0}.
\]

\texttt{(S:FuncSetup $\alpha$)} is the assumption of the statement, placed before the colon. 
\texttt{(S.idealFamily).NDS <= 0} is the conclusion of the statement.
The codes followed by ``\texttt{:=}'' are the proof of the statement. Proof codes consist of tactics in Lean 4.
Theorem \verb+theorem main_nds_nonpos_of_secondary+ in the repository is the reduction theorem for proving the Main Theorem from the Secondary Main Theorem.

% --- Secondary Main Theorem ---
\subsection*{Secondary Main Theorem}
\begin{lstlisting}
theorem secondary_main_theorem {$\alpha$ : Type}
  (S : FuncSetup $\alpha$) (hpos : isPoset S) :
  (S.idealFamily).NDS $\leq$ 0
\end{lstlisting}
Assumption \texttt{isPoset S} means 
%that the order \texttt{FuncSetup.le} $S$ is a partial order. 
that the order relation of $S$ is a partial order.
This statement is the Secondary Main Theorem: for any functional poset (rooted forest poset), the induced order ideal family is always average-rare.
This assumption is stronger than that of the Main Theorem. The Secondary Main Theorem is also used as a lemma for proving the Main Theorem.

\begin{comment}
% --- TraceAt ---
\subsection*{Trace Operator}
\begin{lstlisting}
def traceAt (x : $\alpha$) (F : SetFamily $\alpha$) : SetFamily $\alpha$ :=
  { ground := F.ground.erase x,
    sets := fun A => (F.sets (A $\cup$ {x}) $\lor$ F.sets A)
  }

\end{lstlisting}
This code defines the trace operator for a set family, which removes an element \(x\) from all sets in the family:
\[
\operatorname{trace}_x(\mathcal{F}) := \{ F \setminus \{x\} \mid F \in \mathcal{F} \}.
\]
The trace operator is used to remove an element from all sets, essential for inductive proofs and structural reductions.
\end{comment}

\bigskip
% --- AI-assisted formalisation ---
During the formalization, we made extensive use of AI-assisted tools, notably ChatGPT~5 (for brainstorming tactics and diagnosing type errors), Lean Copilot (for interactive tactic completion), and GitHub Copilot (for boilerplate suggestions). These tools drastically reduced the development time, from what would have taken months by hand to only a few weeks, while correctness is uncompromised since all suggested fragments are checked by Lean’s kernel.
Our experience shows that recent AI tools have made Lean 4 formalization of advanced combinatorial mathematics practically feasible. Without them, the scale of the present project would likely have been out of reach. We therefore regard this as evidence that AI-assisted proof development is becoming an essential methodology in formal mathematics.

% ---------------------------------------------------------------------------
% --- Conclusion and Future Directions -------------------------------------
\section{Conclusion and Future Directions}\label{sec:conclusion}

In this paper, we have proved that every functional preorder, a preorder generated by a function $f : V \to V$, induces a family of order ideals that is average-rare. 
% This result confirms Frankl's Conjecture for the subclass of closure systems in which each vertex serves as the root of exactly one stem of size one in a rooted-set representation. The proof relies on a key reduction: each non-trivial equivalence class is replaced by a single vertex via the trace operation, which never increases the normalized degree sum~$\nds$. The Secondary Main Theorem then establishes $\nds(\mathcal{I}) \le 0$ for rooted-forest posets by induction on the number of vertices. 
% Extending this approach to rooted sets with larger stems would represent a natural next step toward broader cases of the conjecture.
Average-rarity is a stronger condition than the existence of a rare element.
The existence of a rare element in the order-ideal family of a preordered set in general can be shown rather easily (Lemma~\ref{lem:rare}). That is, Frankl's conjecture holds for the order-ideal family of any preordered set. On the other hand, average-rarity does not hold for the order-ideal family of any preordered set, as observed in Example~\ref{example:not_average_rare}.
In this paper we confirmed the average-rarity for the order-ideal family of a functional preordered set (Main Theorem (Theorem~\ref{thm:main})).
This includes the case for functional posets, equivalently, for rooted-forest posets (Secondary Main Theorem (Theorem~\ref{thm:sub})).

The order-ideal family of a preordered set is equivalently rephrased as an intersection-closed family represented by rooted sets such that
the stems are singletons (Lemma~\ref{lem:preorder-ideal}).
Further, if the preorder is functional, then
each element is a root of at most one rooted set
in the representation. 
Extending this approach, a natural next step toward broader cases is to consider rooted sets with stems of larger size. We have the following conjecture.

\begin{conjecture}%[Roots-per-vertex conjecture]
Let $\mathcal C$ be an intersection-closed family on a finite ground set~$V$
that contains $\varnothing$ and admits a rooted-set representation in which
every element of~$V$ is the root of 
% exactly one 
at most one rooted set.  Then
$\mathcal C$ is average-rare.
\end{conjecture}
%In this conjecture, if there exists an element such that it is the root of no rooted set, then such an element is shown to be rare. Hence, it is enough to consider the case in which each element is the root of exactly one rooted set.

In this setting of the conjecture,
% without the condition that the stems are singletons, 
the existence of a rare element, that is, Frankl's conjecture for this case, is still an open problem.
A positive answer to the conjecture would confirm Frankl’s conjecture for this special case.
Under the condition of the conjecture, if there exists an element such that it is the root of no rooted set, then such an element is easily shown to be rare. Hence, it is enough to consider the case in which each element is the root of exactly one rooted set if one aims to verify Frankl's conjecture for this case.

%\paragraph{Formal verification outlook.}
%Our complete \textsc{Lean 4} development \cite{KashiwabaraRepo2025}

%(\url{https://github.com/kashiwabarakenji/ideal_frankl}) is intended as a
%reusable foundation.  
%Two directions appear promising:
%\begin{itemize}
%  \item \emph{Automated search for counterexamples.}  
%The Lean code already enumerates small functional graphs; extending it to stem-size~2  families may reveal extremal configurations.
%  \item \emph{Entropy-style bounds in Lean.}  Recent entropy proofs
%        \cite{Gilmer2022,sawin2022improved} have yet to be formalised.
%        Integrating them with our combinatorial framework could lead to a fully
%        verified lower-bound hierarchy.
%\end{itemize}

%We hope that both the mathematical results and their formal proofs stimulate
%further interaction between extremal combinatorics and proof assistants.
% ---------------------------------------------------------------------------

\end{document}